\documentclass[final]{siamltex}

\usepackage{graphicx}
\usepackage{color}
\usepackage{bm}
\usepackage{amsfonts}
\usepackage{amsmath}
\usepackage{multirow}

\newcommand{\tr}{^{\sf T}}
\newcommand{\m}[1]{{\bf{#1}}}
\newcommand{\g}[1]{\bm #1}
\newcommand{\C}[1]{{\cal {#1}}}

\title{
Convergence rate for a Radau collocation method
applied to unconstrained optimal control
\thanks{
August 17, 2015.
Revised September 12, 2015.
The authors gratefully acknowledge support by
the Office of Naval Research under grants N00014-11-1-0068 and
N00014-15-1-2048, and by the National Science Foundation under
grants DMS-1522629 and CBET-1404767.}
}

\author{
William W. Hager\thanks{{\tt hager@ufl.edu},
http://people.clas.ufl.edu/hager/,
PO Box 118105,
Department of Mathematics,
University of Florida, Gainesville, FL 32611-8105.
Phone (352) 294-2308. Fax (352) 392-8357.}
\and
Hongyan Hou\thanks{{\tt hongyan388@gmail.com},
        Chemical Engineering,
        Carnegie Mellon University,
        5000 Forbes Avenue, Pittsburgh, PA 15213.}
\and
Anil V. Rao\thanks{{\tt anilvrao@ufl.edu},
        http://www.mae.ufl.edu/rao
        Department of Mechanical and Aerospace Engineering,
        P.O. Box 116250, Gainesville, FL 32611-6250.
        Phone:(352) 392-0961. Fax:(352) 392-7303}
}

\begin{document}
\maketitle

\begin{abstract}
A local convergence rate is established for an orthogonal collocation
method based on Radau quadrature applied to an unconstrained optimal
control problem.
If the continuous problem has a sufficiently smooth solution
and the Hamiltonian satisfies a strong convexity condition,
then the discrete problem possesses a local minimizer
in a neighborhood of the continuous solution, and as the number
of collocation points increases, the discrete solution convergences
exponentially fast in the sup-norm to the continuous solution.
An earlier paper analyzes an orthogonal collocation method based on
Gauss quadrature, where neither end point of the problem domain is a
collocation point.
For the Radau quadrature scheme, one end point is a collocation point.
\end{abstract}


\pagestyle{myheadings} \thispagestyle{plain}
\markboth{W. W. HAGER, H. HOU, AND A. V. RAO}
{RADAU COLLOCATION}

\section{Introduction}
A convergence rate is established for an orthogonal collocation method
applied to an unconstrained control problem of the form
\begin{equation}\label{P}
\begin{array}{clc}
\mbox {minimize} &C(\m{x}(1))&\\
\mbox {subject to} &\dot{\m{x}}(t)=
\m{f(x}(t), \m{u}(t)),&t\in[-1,1],\\
&\m{x}(-1)=\m{x}_0,&
\end{array}
\end{equation}
where the state ${\m x}(t)\in \mathbb{R}^n$,
$\dot{\m x}\equiv\displaystyle\frac{d}{dt}{\m x}$,
the control ${\m u}(t)\in {\mathbb R}^m$,
${\m f}: {\mathbb R}^n \times {\mathbb R}^m\rightarrow {\mathbb R}^n$,
$C: {\mathbb R}^n \rightarrow {\mathbb R}$,
and ${\m x}_0$ is the initial condition, which we assume is given.
Assuming the dynamics $\dot{\m{x}}(t)= \m{f(x}(t), \m{u}(t))$ is nice enough,
we can solve for the state $\m{x}$ as a function of the control $\m{u}$,
and the control problem reduces to an unconstrained minimization over $\m{u}$.

Let $\C{P}_N$ denote the space of polynomials of degree at most $N$
defined on the interval $[-1, +1]$, and let $\C{P}_N^n$ denote the
$n$-fold Cartesian product $\C{P}_N \times \ldots \times \C{P}_N$.
We analyze a discrete approximation to (\ref{P}),
introduced in \cite{GargHagerRao11a,GargHagerRao10a}, of the form
\begin{equation}\label{D}
\begin{array}{cl}
\mbox {minimize} &C(\m{x}(1))\\
\mbox {subject to} &\dot{\m{x}}(\tau_i)=
\m{f(x}(\tau_i), \m{u}_i),\quad 1 \le i \le N,\\
&\m{x}(-1)=\m{x}_0, \quad \m{x} \in \C{P}_N^n.
\end{array}
\end{equation}
At the collocation points $\tau_i$, $1 \le i \le N$, the
equation should be satisfied.
The control approximation at time $\tau_i$ is $\m{u}_i$.
We focus on the Radau quadrature points satisfying
\[
-1 < \tau_1 < \tau_2 < \ldots < \tau_N = +1 .
\]
The analysis we give also applies to the flipped Radau scheme
obtained by reversing the sign for each collocation point.
In (\ref{D}) the dimension of $\C{P}_N$ is $N+1$, while there are
$N+1$ equations in (\ref{D}) corresponding to the
collocated dynamics at $N$ points and the initial condition.
When the discrete dynamics is nice enough,
we can solve for the discrete state $\m{x} \in \C{P}_N^n$ as a function
of the discrete controls $\m{u}_i$, $1 \le i \le N$, and the discrete
approximation reduces to an unconstrained minimization over the discrete
controls.

In an earlier paper \cite{HagerHouRao15b} we analyzed a scheme based
on Gauss quadrature, where the collocation points lie in the interior
of the interval $[-1, +1]$.
The Gauss scheme is easier to analyze than the Radau scheme of this
paper due to the symmetry of the Gauss collocation points,
and the fact that the none of the collocation points lies at
an end of the problem domain.
For the Radau scheme,
symmetry is broken and the polynomials in the
discrete adjoint equation have degree $N-1$ compared to the degree $N$
polynomials used for the state approximation.
As will be seen, the presence of the Radau collocation point
at the end of the interval $[-1, +1]$ leads to the embedding of
the terminal adjoint condition of the discrete problem
into the discrete adjoint dynamics.
Despite these differences in the analysis and despite the fact that
Gauss quadrature has a higher degree of accuracy than Radau
quadrature, the convergence rate obtained for the
Radau scheme is exactly the same as that of the Gauss scheme.
Moreover, in numerical experiments with test problems where an
exact solution is known, the observed error in the state, control, and adjoint
for the Radau scheme is very similar to the observed error for the Gauss scheme.
The fact that a Radau quadrature point can be placed at the end point
of the interval leads to a simpler implementation of terminal constraints
and terminal cost.
And the Radau scheme is easier than the Gauss scheme to extend to an
$hp$-framework
\cite{DarbyHagerRao11,DarbyHagerRao10,LiuHagerRao15,PattersonHagerRao14}
where the interval $[-1, +1]$ is partitioned into a mesh
and a different polynomial is employed in each mesh interval.

The analysis in this and the earlier paper \cite{HagerHouRao15b}
needs further extensions in order to handle Lobatto collocation schemes
such as those in \cite{Elnagar1,Fahroo2} where $\tau_1 = -1$ and $\tau_N = +1$.
Although the Gauss and Radau scheme have similar errors,
the Lobatto scheme can converge at
a slower rate, as observed in \cite{GargHagerRao10a}.
An advantage of the Lobatto scheme is that the value of the optimal control
is estimated at the initial point $t = -1$ and the terminal point $t = +1$.
Moreover, both initial and terminal constraints are easier to implement
in the Lobatto framework.
A consistency result for a Lobatto collocation scheme applied to optimal
control is given in \cite{GongRossKangFahroo08}.
Other quadrature points that have been exploited in the optimal
control literature include the Chebyshev quadrature points
\cite{Elnagar4, FahrooRoss02}, and the extrema of Jacobi polynomials
\cite{Williams1}.

Our goal is to show that if $(\m{x}^*, \m{u}^*)$ is a local minimizer
for (\ref{P}), then the discrete problem (\ref{D}) has a local minimizer
$(\m{x}^N, \m{u}^N)$ that converges exponentially fast in $N$ to
$(\m{x}^*, \m{u}^*)$ at the collocation points.
Convergence rates have been obtained previously when the approximating
space consists of piecewise polynomials as in
\cite{DontchevHager93, DontchevHager97,DontchevHagerVeliov00,
DontchevHagerMalanowski00, Hager99c, Kameswaran1, Reddien79}.
In these earlier results,
convergence is achieved by letting the mesh spacing tend to zero.
In our results, on the other hand, convergence is achieved by letting
$N$, the degree of the approximating polynomials, tend to infinity.

Let $\C{C}^k (\mathbb{R}^n)$ denote the space of $k$ times
continuously differentiable functions
$\m{x}: [-1, +1] \rightarrow \mathbb{R}^n$ with the sup-norm
$\| \cdot \|_\infty$ given by
\begin{equation}\label{csup}
\|\m{x}\|_\infty = \sup \{ |\m{x} (t)| : t \in [-1, +1] \} ,
\end{equation}
where $| \cdot |$ is the Euclidean norm.
It is assumed that (\ref{P}) has a local minimizer
$(\m{x}^*, \m{u}^*)$ in $\C{C}^1 (\mathbb{R}^n) \times \C{C}^0 (\mathbb{R}^m)$.
Given $\m{y} \in \mathbb{R}^n$, the ball with center $\m{y}$ and radius
$\rho$ is denoted
\[
\C{B}_\rho (\m{y}) = \{ \m{x} \in \mathbb{R}^n :
|\m{x} - \m{y}| \le \rho \} .
\]
It is assumed that there exists an open set
$\Omega \subset \mathbb{R}^{m+n}$ and $\rho > 0$ such that
\[
\C{B}_\rho (\m{x}^*(t),\m{u}^*(t)) \subset \Omega \mbox{ for all }
t \in [-1, +1].
\]
Moreover, the first two derivative of $f$ and $C$ are
continuous on the closure of
$\Omega$ and on $\C{B}_\rho (\m{x}^*(1))$ respectively.

Let $\g{\lambda}^*$ denote the solution of the
linear costate equation
\begin{equation}\label{costate}
\dot{\g{\lambda}}^*(t)=-\nabla_xH({\m x}^*(t), {\m u}^*(t), {\g \lambda}^*(t)),
\quad {\g \lambda}^*(1)=\nabla C({\m x}^*(1)),
\end{equation}
where $H$ is the Hamiltonian defined by
$H({\m x}, {\m u}, {\g \lambda}) ={\g\lambda}\tr {\m f}({\m x}, {\m u})$.
Here $\nabla C$ denotes the gradient of $C$.
By the first-order optimality conditions (Pontryagin's minimum principle),
we have
\begin{equation} \label{controlmin}
\nabla_u H({\m x}^*(t), {\m u}^*(t), {\g \lambda}^*(t)) = \m{0}
\end{equation}
for all $t \in [-1, +1]$.

Since the discrete collocation problem (\ref{D}) is finite dimensional,
the first-order optimality conditions (Karush-Kuhn-Tucker conditions)
imply that when a constraint qualification holds
\cite{NocedalWright2006}, the gradient of the Lagrangian vanishes.
By the analysis in \cite{GargHagerRao11a, GargHagerRao10a},
the gradient of the Lagrangian vanishes if and only if there exists
$\g{\lambda} \in \C{P}_{N-1}^n$ such that
\begin{eqnarray}
\g{\lambda}_0 &=& \g{\lambda}(-1)
\label{dstart} \\[.05in]
\dot{\g \lambda}(\tau_i) &=&
-\nabla_x H\left({\m x} (\tau_i),{\m u}_i, {\g \lambda} (\tau_i) \right),
\quad 1 \leq i < N,  \label{dcostate} \\[.05in]
\dot{\g \lambda}(1) &=&
-\nabla_x H\left({\m x} (1),{\m u}_N, {\g \lambda} (1) \right)
+ (\g{\lambda}(1) - \nabla C (\m{x}(1))/\omega_N,
\label{dterminal}\\[.05in]
\m{0} &=&
\nabla_u H\left({\m x}(\tau_i),{\m u}_i, {\g \lambda} (\tau_i) \right),
\quad 1\leq i\leq N, \label{dcontrolmin}
\end{eqnarray}
where $\omega_i$ is the Radau quadrature weight associated with $\tau_i$, and
$\g{\lambda}_0$ is the multiplier associated with the initial condition
$\m{x}(-1) = \m{x}_0$ in (\ref{D}).
Note that in \cite{GargHagerRao10a},
(\ref{dstart}) is written in the form
\[
\nabla C(\m{x}(1)) = \g{\lambda}_0 - \sum_{i=1}^N
\omega_i \nabla_x H({\m x} (\tau_i),{\m u}_i, {\g \lambda} (\tau_i) ) .
\]
However, utilizing (\ref{dcostate}), (\ref{dterminal}),
and the fundamental theorem of calculus, this reduces to the more
compact form (\ref{dstart}).

In comparing the first-order conditions for Radau collocation to the
first-order conditions for Gauss collocation \cite{HagerHouRao15b},
the differences are that in Gauss collocation,
$\g{\lambda} \in \C{P}_{N}^n$ not $\C{P}_{N-1}^n$.
Moreover, in Gauss collocation, the terminal condition for the
discrete adjoint is simply $\g{\lambda}(1) = \nabla C(\m{x}(1))$,
while in Radau collocation, the discrete costate dynamics and the terminal
condition are mixed together as in (\ref{dterminal}).

The assumptions utilized in the convergence analysis are the following:

\begin{itemize}
\item[(A1)]
$\m{x}^*$ and $\g{\lambda}^* \in \C{C}^{\eta+1}$ for some $\eta \ge 3$.
\item[(A2)]
For some $\alpha > 0$,
the smallest eigenvalue of the Hessian matrices
\[
\nabla^2 C(\m{x}^*(1)) \quad \mbox{and} \quad
\nabla^2_{(x,u)} H(\m{x}^* (t),
\m{u}^* (t), \g{\lambda}^* (t) )
\]
is greater than $\alpha$, uniformly for $t \in [-1, +1]$.
\item[(A3)]
The Jacobian of the dynamics satisfies
\[
\|\nabla_x \m{f} (\m{x}^*(t), \m{u}^*(t))\|_\infty \le 1/4
\quad \mbox{and} \quad
\|\nabla_x \m{f} (\m{x}^*(t), \m{u}^*(t))\tr\|_\infty \le 1/4
\]
for all $t \in [-1, +1]$ where $\| \cdot \|_\infty$ is the matrix
sup-norm (largest absolute row sum), and the Jacobian
$\nabla_x \m{f}$ is an $n$ by $n$  matrix whose $i$-th row is
$(\nabla_x f_i)\tr$.
\end{itemize}
\smallskip

The smoothness assumption (A1) is used to obtain a bound for the
accuracy with which the interpolant of the continuous state $\m{x}^*$
satisfies the discrete dynamics.
The coercivity assumption (A2) ensures that the solution of the
discrete problem is a local minimizer.
The condition (A3) enters into the analysis of stability for
the perturbed dynamics;
this condition can be eliminated when the problem domain $[-1, +1]$
is partitioned into subintervals with a different polynomial on each
subinterval
\cite{DontchevHager93, DontchevHager97, DontchevHagerMalanowski00,
DontchevHagerVeliov00, Hager99c, Kameswaran1, Reddien79}.
For the global polynomials analyzed in this paper,
(A3) could be replaced by any condition that ensures
stability of the linearized dynamics.

In addition to the 3 assumptions, the analysis utilizes 4 properties
of the Radau collocation scheme.
Let $\tau_0 = -1$, a noncollocated point, and define
\begin{equation}\label{Ddef}
D_{ij} = \dot{L}_j (\tau_i), \;
1 \le i \le N, \;
0 \le j \le N, \quad
\mbox{where }
L_j (\tau) := \prod^{N}_{\substack{i=0\\i\neq j}}
\frac{\tau-\tau_i}{\tau_j-\tau_i} .
\end{equation}
Here the dot denotes differentiation, and
$\m{D}$ is a differentiation matrix in the sense that
$(\m{Dp})_i = \dot{p} (\tau_i)$, $1 \le i \le N$,
whenever $p \in \C{P}_N$ is the polynomial that satisfies
$p(\tau_j) = p_j$ for $0 \le j \le N$.
The submatrix $\m{D}_{1:N}$ consisting of the tailing $N$ columns of $\m{D}$
has the following properties:
\smallskip
\begin{itemize}
\item [(P1)]
$\m{D}_{1:N}$ is invertible and
$\| \m{D}_{1:N}^{-1}\|_\infty = 2$.
\item [(P2)]
If $\m{W}$ is the diagonal matrix containing the Radau
quadrature weights $\g{\omega}$ on the diagonal, then the rows of the
matrix $[\m{W}^{1/2} \m{D}_{1:N}]^{-1}$ have Euclidean norm bounded by
$\sqrt{2}$.
\end{itemize}
\smallskip
The fact that $\m{D}_{1:N}$ is invertible is established in
\cite[Prop. 1]{GargHagerRao11a}, and a formula for the elements of
$\m{D}_{1:N}^{-1}$ is given in \cite[equation (53)]{GargHagerRao10a}.
From the formula, the elements in the last row of
$\m{D}_{1:N}^{-1}$ are the Radau quadrature weights, which are
positive and sum to 2.
Although elements in the earlier rows of $\m{D}_{1:N}^{-1}$ can
be either positive or negative, we find numerically that their
absolute sum is always less than 2.
Similarly, the elements in the last row of
$[\m{W}^{1/2} \m{D}_{1:N}]^{-1}$ are the square roots of
the Radau quadrature weights.
Hence, the Euclidean norm of the last row of
$[\m{W}^{1/2} \m{D}_{1:N}]^{-1}$ is $\sqrt{2}$.
Numerically, we find that Euclidean norm of the earlier rows is
always less than $\sqrt{2}$.

Let $\m{D}^\ddagger$ by the $N$ by $N$ matrix defined by
\[
D_{ij}^\ddagger = - \left( \frac{\omega_j}{\omega_i} \right)
D_{ji}, \quad 1 \le i \le N, \quad 1 \le j \le N .
\]
The matrix $\m{D}^\ddagger$ arises in the analysis of the costate equation.
In Section~4.2.1 of \cite{GargHagerRao10a}, we introduce a matrix
$\m{D}^\dagger$ which is a differentiation matrix for the collocation
points $\tau_i$, $1 \le i \le N$.
That is, if $p$ is a polynomial of degree at most $N-1$ and
$\m{p}$ is the vector with components $p(\tau_i)$, $1 \le i \le N$,
then $(\m{D}^\dagger \m{p})_i = \dot{p}(\tau_i)$.
The matrix $\m{D}^\ddagger$ only differs from $\m{D}^\dagger$ in a
single entry: $D^\ddagger_{NN} = D^\dagger_{NN} - 1/\omega_N$.
As a result,
\begin{equation}\label{h282}
(\m{D}^{\ddag} \m{p})_i = \dot{p}(\tau_i), \quad
1 \le i < N, \quad
(\m{D}^{\ddag} \m{p})_N = \dot{p}(\tau_N) - p(1)/\omega_N.
\end{equation}
If $\m{D}^{\ddag} \m{p} = \m{0}$, then
$\dot{p}(\tau_i) = 0$ for $i < N$ by (\ref{h282}).
Since $\dot{p}$ has degree $N-2$ and it vanishes at $N-1$ points,
$\dot{p}$ is identically zero and $p$ is constant.
By the final equation in (\ref{h282}), $p(1) = 0$ when
$\m{D}^\ddag \m{p} = \m{0}$, which implies that $p$ is identically zero.
This shows that $\m{D}^\ddagger$ is invertible.
We find that $\m{D}^\ddagger$ has the following properties:
\smallskip
\begin{itemize}
\item [(P3)]
$\m{D}^\ddagger$ is invertible and
$\| (\m{D}^\ddagger)^{-1}\|_\infty \le 2$.
\item [(P4)]
The rows of the
matrix $[\m{W}^{1/2} \m{D}^\ddagger]^{-1}$ have Euclidean norm bounded by
$\sqrt{2}$.
\end{itemize}
\smallskip
In Proposition~\ref{deriv_exact} at the end of the paper,
an explicit formula is given for the inverse of $\m{D}^\ddagger$.
However, it is not clear from the formula that
$\| (\m{D}^\ddagger)^{-1}\|_\infty$ is bounded by 2.
Numerically, we find that the norms in (P3) and (P4) achieve
their maximum in the first row of the matrix,
and these norms increase monotonically towards the given bounds.
Properties (P1)--(P4) differ from the assumptions (A1)--(A3)
in the sense that the properties seem to hold for any choice of $N$,
although a proof is missing, while (A1)--(A3) only hold
for certain control problems.
In the analysis of the Gauss scheme \cite{HagerHouRao15b},
properties (P3) and (P4) followed immediately from (P1) and (P2)
since it could be shown that the discrete costate matrix was related
to the state differentiation matrix through an exchange operation.
However, due to the asymmetry of the Radau collocation points and the
lower degree of the polynomials in the discrete adjoint system
(\ref{dstart})--(\ref{dcontrolmin}), the relation between the
state and costate matrices for the Radau scheme is not clear.
Nonetheless, the bounds in (P3) and (P4) are observed to be the same as the
bounds in (P1) and (P2).

If $\m{x}^N$ is a solution of (\ref{D}) associated
with the discrete controls $\m{u}_i$, $1 \le i \le N$, and
if $\g{\lambda}^N \in \C{P}_{N-1}^n$ satisfies
(\ref{dstart})--(\ref{dcontrolmin}), then we define
\[
\begin{array}{lllllll}
\m{X}^N &= [ &\m{x}^N(\tau_0), & \m{x}^N(\tau_1), & \ldots,
& \m{x}^N(\tau_N) &], \\
\m{X}^* &= [ &\m{x}^*(\tau_0), & \m{x}^*(\tau_1), & \ldots,
& \m{x}^*(\tau_N) &], \\
\m{U}^N &= [ && \m{u}_1, & \ldots, & \m{u}_N &], \\
\m{U}^* &= [ && \m{u}^*(\tau_1), & \ldots, & \m{u}^*(\tau_N)& ], \\
\g{\Lambda}^N &= [ &\g{\lambda}^N(\tau_0), & \g{\lambda}^N(\tau_1),
& \ldots, & \g{\lambda}^N(\tau_N) &], \\
\g{\Lambda}^* &= [ &\g{\lambda}^*(\tau_0), & \g{\lambda}^*(\tau_1),
& \ldots, & \g{\lambda}^*(\tau_N) &].
\end{array}
\]
For any of the discrete variables, we define a discrete
sup-norm analogous to the continuous sup-norm in (\ref{csup}).
For example, if $\m{U}^N \in \mathbb{R}^{mN}$ with $\m{U}_i \in \mathbb{R}^m$,
then
\[
\|\m{U}^N\|_\infty = \sup \{ |\m{U}_i| : 1 \le i \le N \}.
\]
The following convergence result is established:
\smallskip
\begin{theorem}\label{maintheorem}
If $(\m{x}^*, \m{u}^*)$ is a local minimizer for the continuous problem
$(\ref{P})$ and both {\rm (A1)--(A3)} and {\rm (P1)--(P4)} hold,
then for $N$ sufficiently large with $N > \eta+1$,
the discrete problem $(\ref{D})$ has a
local minimizer $(\m{X}^N, \m{U}^N)$ for which
\begin{equation}\label{maineq}
\max \left\{ \|{\bf X}^N-{\bf X}^*\|_\infty ,
\|{\bf U}^N-{\bf U}^*\|_\infty,
\|{\g \Lambda}^N-{\g \Lambda}^*\|_\infty \right\} \leq cN^{2-\eta},
\end{equation}
where $c$ is independent of $N$.
\end{theorem}
\smallskip

The discrete problem provides an estimate for optimal control at
$\tau_N = +1$, however, there is no discrete control at $\tau_0 = -1$.
Due to the strong convexity assumption (A2),
an estimate for the discrete control at $t = -1$ can be obtained from
the minimum principle (\ref{controlmin}) since we have estimates
for the discrete state and costate at $\tau_0 = -1$.
Alternatively, polynomial interpolation could be used to obtain
estimates for the optimal control at $t = -1$.

The paper is organized as follows.
In Section~\ref{abstract} the discrete optimization
problem (\ref{D}) is reformulated as a nonlinear system of equations
obtained from the first-order optimality conditions,
and a general approach to convergence analysis is presented.
Section~\ref{residual} obtains an estimate for how closely the
solution to the continuous problem satisfies the first-order optimality
conditions for the discrete problem.
Section~\ref{inverse} proves that the linearization of the discrete
control problem around a solution of the continuous problem is invertible.
Section~\ref{omegabounds} establishes an $L^2$ stability property
for the linearization, while
Section~\ref{inftybounds} strengthens the norm to $L^\infty$.
This stability property is the basis for the proof of Theorem~\ref{maintheorem}.
A numerical example illustrating the exponential convergence
result is given in Section~\ref{numerical}.

{\bf Notation.}
The meaning of the norm $\| \cdot \|_\infty$ is based on context.
If $\m{x} \in \C{C}^0 (\mathbb{R}^n)$, then
$\|\m{x}\|_\infty$ denotes the maximum of $|\m{x}(t)|$ over
$t \in [-1, +1]$, where $| \cdot|$ is the Euclidean norm.
If $\m{A} \in \mathbb{R}^{m \times n}$, then $\|\m{A}\|_\infty$
is the largest absolute row sum (the matrix norm induces by the
$\ell_\infty$ vector norm).
If $\m{U} \in \mathbb{R}^{mN}$ is the discrete control with
$\m{U}_i \in \mathbb{R}^m$, then $\|\m{U}\|_\infty$ is the maximum
of $|\m{U}_i|$, $1 \le i \le N$.
The dimension of the identity matrix $\m{I}$ is often clear from context;
when necessary, the dimension of $\m{I}$ is specified by a subscript.
For example, $\m{I}_n$ is the $n$ by $n$ identity matrix.
$\nabla C$ denotes the gradient, a column vector, while
$\nabla^2 C$ denotes the Hessian matrix.
Throughout the paper, $c$ denotes a generic constant which has different
values in different equations.
The value of this constant is always independent of $N$, the degree
of the polynomials used in the discrete approximation of the state.
$\m{1}$ denotes a vector whose entries are all equal to one, while
$\m{0}$ is a vector whose entries are all equal to zero, their
dimension should be clear from context.
If $\m{D}$ is the differentiation matrix introduced in (\ref{Ddef}), the
$\m{D}_j$ is the $j$-th column of $\m{D}$ and
$\m{D}_{i:j}$ is the submatrix formed by columns $i$ through $j$.

\section{Abstract setting}
\label{abstract}
Given $\m{x} \in \C{P}_N^n$ and $\m{u} \in \mathbb{R}^{mN}$
that are feasible in (\ref{D}),
define $\m{X}_i = \m{x}(\tau_i)$ and $\m{U}_i = \m{u}_i$.
As shown in \cite{GargHagerRao10a}, the discrete problem (\ref{D})
can be reformulated as the nonlinear programming problem
\begin{equation}\label{nlp}
\begin{array}{ll}
\mbox {minimize} &C(\m{X}_{N})\\[.05in]
\mbox {subject to} &\sum_{j=0}^{N}{D}_{ij}{\m X}_j
={\m f}({\m X}_i,{\m U}_i), \quad 1\leq i\leq N,\\[.05in]
&\m{X}_0=\m{x}_0.
\end{array}
\end{equation}
Also, \cite{GargHagerRao10a} shows that the equations obtained by
setting the gradient of the Lagrangian to zero are equivalent
to the system of equations
\begin{eqnarray}
\g{\Lambda}_0 &=& \nabla C(\m{X}_N) + \sum_{i=1}^N
\omega_i \nabla_x H(\m{X}_i, \m{U}_i, \g{\Lambda}_i), \label{Dstart} \\
\sum_{j=1}^{N}{D}_{ij}^\ddagger{\g \Lambda}_j &=&
-\nabla_x H\left({\m X}_i,{\m U}_i, {\g \Lambda}_i\right),
\quad 1 \leq i < N, \label{Dadjoint} \\
\sum_{j=1}^{N}{D}_{Nj}^\ddagger{\g \Lambda}_j &=&
-\nabla_x H\left({\m X}_N,{\m U}_N, {\g \Lambda}_N\right)
-\nabla C (\m{X}_N)/\omega_N,
\label{D_1adjoint}\\
\m{0} &=&
\nabla_u H\left({\m X}_i,{\m U}_i, {\g \Lambda}_i\right),
\quad 1\leq i\leq N,  \label{Dcontrolmin}
\end{eqnarray}
where ${\g \Lambda_0}$ is the multiplier associated with the equation
${\m X}_0 = \m{x}_0$ and $\g{\Lambda}_i$ for $i > 0$ is related to the
Lagrange multiplier $\g{\lambda}_i$ associated with the $i$-th
equation in the discrete dynamics by
\begin{equation}\label{eq12}
\g{\Lambda}_i = \g{\lambda}_i /\omega_i .
\end{equation}

The first-order optimality conditions for the nonlinear program (\ref{nlp})
consist of the equations (\ref{Dstart})--(\ref{Dcontrolmin}),
and the constraints in (\ref{nlp}).
This system can be written as $\C{T}(\m{X}, \m{U}, \g{\Lambda}) = \m{0}$ where
\[
(\C{T}_1, \C{T}_2, \ldots, \C{T}_6) (\m{X}, \m{U}, \g{\Lambda}) \in
\mathbb{R}^{nN} \times \mathbb{R}^n \times \mathbb{R}^n \times
\mathbb{R}^{n(N-1)} \times \mathbb{R}^n \times \mathbb{R}^{mN}.
\]
The 6 components of $\C{T}$ are defined as follows:
\begin{eqnarray*}
\C{T}_{1i}(\m{X}, \m{U}, \g{\Lambda}) &=&
\left( \sum_{j=0}^{N}{D}_{ij}{\bf X}_j \right) -{\bf f}({\bf X}_i,{\bf U}_i),
\quad 1\leq i\leq N ,\\
\C{T}_2(\m{X}, \m{U}, \g{\Lambda}) &=&
{\bf X}_{0}-{\m x}_0, \\
\C{T}_{3}(\m{X}, \m{U}, \g{\Lambda}) &=&
\g{\Lambda}_0 - \nabla C(\m{X}_N) - \sum_{i=1}^N
\omega_i \nabla_x H(\m{X}_i, \m{U}_i, \g{\Lambda}_i),\\
\C{T}_{4i}(\m{X}, \m{U}, \g{\Lambda}) &=&
\left( \sum_{j=1}^{N}{D}_{ij}^\ddagger{\g \Lambda}_j \right) +
\nabla_x H({\bf X}_i,{\bf U}_i, {\g \Lambda}_i),
\quad 1 \leq i < N , \\
\C{T}_5(\m{X}, \m{U}, \g{\Lambda}) &=&
\sum_{j=1}^{N}{D}_{Nj}^\ddagger{\g \Lambda}_j
+\nabla_x H\left({\m X}_N,{\m U}_N, {\g \Lambda}_N\right)
+\nabla C (\m{X}_N)/\omega_N,\\
\C{T}_{6i}(\m{X}, \m{U}, \g{\Lambda}) &=&
\nabla_u H({\bf X}_i, {\bf U}_i, {\g \Lambda}_i),
\quad 1\leq i\leq N .
\end{eqnarray*}

The proof of Theorem~\ref{maintheorem} reduces to a study of solutions
to $\C{T}(\m{X}, \m{U}, \g{\Lambda}) = \m{0}$ in a neighborhood of
$(\m{X}^*, \m{U}^*, \g{\Lambda}^*)$.
Our analysis is based on
\cite[Proposition 3.1]{DontchevHagerVeliov00}, which we simplify
below to take into account the structure of our $\C{T}$.
Other results like this are contained in
Theorem~3.1 of \cite{DontchevHager97},
in Proposition~5.1 of \cite{Hager99c}, and in Theorem~2.1 of \cite{Hager02b}.
\begin{proposition}\label{prop}
Let $\mathcal{X}$ be a Banach space and $\mathcal{Y}$
be a linear normed space with the norms in both spaces denoted
$\|\cdot\|$.
Let $\mathcal{T}$: $\mathcal{X} \longmapsto \mathcal{Y}$ with
$\mathcal{T}$ continuously Fr\'{e}chet differentiable in
$\C{B}_r(\g{\theta}^*)$ for
some $\g{\theta}^* \in \mathcal{X}$ and $r> 0$.
Suppose that
\[
\|\nabla\mathcal{T}(\g{\theta})-\nabla\mathcal{T}(\g{\theta}^*)\|
\leq \varepsilon \mbox{ for all } \g{\theta} \in
{\mathcal B}_r(\g{\theta}^*)
\]
and $\nabla\mathcal{T}(\g{\theta}^*)$ is invertible;
and define $\mu := \|\nabla\mathcal{T}(\g{\theta}^*)^{-1}\|$.
If $\varepsilon\mu < 1$ and
$\left\|\mathcal{T}\left(\g{\theta}^*\right)\right\|
\leq(1-\mu\varepsilon)r/\mu$,
then there exists a unique
$\g{\theta} \in {\mathcal B}_r(\g{\theta}^*)$ such that
$\mathcal{T}(\g{\theta})=\m{0}$.
Moreover, we have the estimate
\begin{equation}\label{abs}
\|\g{\theta}-\g{\theta}^*\|\leq \frac{\mu}{1-\mu\varepsilon}
\left\|\mathcal{T}\left(\g{\theta}^*\right)\right\|.
\end{equation}
\end{proposition}

We apply Proposition~\ref{prop} with
$\g{\theta}^* = (\m{X}^*, \m{U}^*, \g{\Lambda}^*)$ and
$\g{\theta} = (\m{X}^N, \m{U}^N, \g{\Lambda}^N)$.
The key steps in the analysis are the estimation of the residual
$\left\|\mathcal{T}\left(\g{\theta}^*\right)\right\|$,
the proof that
$\nabla\mathcal{T}(\g{\theta}^*)$ is invertible,
and the derivation of a bound for
$\|\nabla\mathcal{T}(\g{\theta}^*)^{-1}\|$ that is independent of $N$.
In our context, the norm on $\C{X}$ is
\begin{equation}\label{Xnorm}
\|\g \theta\|=\|({\bf X},{\bf U}, {\g \Lambda})\|_\infty
=\max\{\|\bf X\|_\infty, \|\bf U\|_\infty,\|\g \Lambda\|_\infty\}.
\end{equation}
For this norm, the left side of (\ref{maineq}) and the left side of (\ref{abs})
are the same.
The norm on $\C{Y}$ enters into the estimation of both the residual
$\|\mathcal{T}(\g{\theta}^*)\|$ in
(\ref{abs}) and the parameter
$\mu := \|\nabla\mathcal{T}(\g{\theta}^*)^{-1}\|$.
In our context,
we think of an element of $\C{Y}$ as a vector with components
$\m{y}_i$, $1 \le i \le 3N + 2$,
where $\m{y}_i \in \mathbb{R}^n$ for $1 \le i \le 2N + 2$ and
$\m{y}_i \in \mathbb{R}^m$ for $i > 2N + 2$.
For example, $\C{T}_1(\m{X},\m{U}, \g{\Lambda}) \in \mathbb{R}^{nN}$
corresponds to the components $\m{y}_i \in \mathbb{R}^n$, $1 \le i \le N$.
For the norm in $\C{Y}$, we take
\begin{equation}\label{Ynorm}
\|\m{y}\|_\infty = \sup \{ |\m{y}_i| : 1 \le i \le 3N + 2 \} .
\end{equation}
%
\section{Analysis of the residual}
\label{residual}
We now establish a bound for $\C{T}(\m{X}^*, \m{U}^*, \g{\Lambda}^*)$,
the residual which appears on the right side of the error bound (\ref{abs}).
This bound for the residual ultimately appears in the right side of
the error estimate (\ref{maineq}).
\smallskip
\begin{lemma}\label{residuallemma}
If {\rm (A1)} holds, then there
exits a constant $c$, independent of $N$, such that
\begin{equation}\label{delta}
\|\mathcal{T}(\m{X}^*, \m{U}^*, \g{\Lambda}^*)\|_\infty \le
c N^{2-\eta}
\end{equation}
for all $N > \eta+1$.
\end{lemma}
\begin{proof}
By the definition of $\C{T}$, we have
\[
\C{T}_2 (\m{X}^*,  \m{U}^*, \g{\Lambda}^*) =
\m{X}^*_0 - \m{x}_0 = \m{x}^*(\tau_0) - \m{x}_0 = \m{x}^*(-1) - \m{x}_0 = \m{0}
\]
since $\m{x}^*$  satisfies the initial condition in (\ref{P}).
Likewise,
$\C{T}_6(\m{X}^*,  \m{U}^*, \g{\Lambda}^*) = \m{0}$ since
(\ref{controlmin}) holds for all $t \in [-1, +1]$, which implies that
(\ref{controlmin}) holds at the collocation points.

Now consider $\C{T}_1$.
Since $\m{D}$ is a differentiation matrix associated with the
collocation points, we have
\begin{equation}\label{h115}
\sum_{j=0}^{N} D_{ij} \m{X}_j^* = \dot{\m{x}}^I(\tau_i), \quad 1 \le i \le N,
\end{equation}
where $\m{x}^I \in \C{P}_N^n$ is the (interpolating) polynomial that passes
through $\m{x}^*(\tau_j)$ for $0 \le j \le N$.
Since $\m{x}^*$ satisfies the dynamics of (\ref{P}),
\begin{equation}\label{h116}
\m{f}(\m{X}_i^*, \m{U}_i^*) =
\m{f}(\m{x}^*(\tau_i), \m{u}^*(\tau_i)) = \dot{\m{x}}^* (\tau_i).
\end{equation}
Combine (\ref{h115}) and (\ref{h116}) to obtain
\begin{equation}\label{T1i}
\C{T}_{1i}(\m{X}^*,  \m{U}^*, \g{\Lambda}^*) =
\dot{\m{x}}^I(\tau_i) - \dot{\m{x}}^*(\tau_i) .
\end{equation}
Proposition~2.1 and Lemma~2.2 in
\cite{HagerHouRao15a} yield
\[
\|\dot{\m{x}}^I - \dot{\m{x}}^*\|_\infty  \le
\left( \frac{6e}{N-1} \right)^\eta
\left[ (1+2N^2) + 6e N (1+c_1 \log N) \right]
\left( \frac{ 12 \|x^{(\eta+1)}\|}{\eta+1} \right)
\]
for all $N > \eta+1$,
where $\m{x}^{(\eta+1)}$ is the $(\eta+1)$-st derivative of $\m{x}$
and $c_1 \log N$ is a bound for the Lebesgue constant of the
point set $\tau_j$, $0 \le j \le N$,
given in Theorem~2.1 of \cite{Vertesi81}.
Hence, there exists a constant $c_2$, independent of $N$ but dependent on
$\eta$, such that
\begin{equation} \label{h2}
|\C{T}_{1i}(\m{X}^*,  \m{U}^*, \g{\Lambda}^*)| =
|\dot{\m{x}}^I(\tau_i) - \dot{\m{x}}^*(\tau_i)| \le
\|\dot{\m{x}}^I - \dot{\m{x}}^*\|_\infty  \le c_2 N^{2-\eta},
\end{equation}
which complies with the bound (\ref{delta}).

Next, let us consider $\C{T}_4$.
By (\ref{h282}) if $\g{\lambda}^I \in \C{P}_{N-1}^n$ is the
(interpolating) polynomial that passes
through $\g{\lambda}^*(\tau_j)$ for $1 \le j \le N$, we have
\begin{equation}
\sum_{j=1}^{N} D_{ij}^\ddagger \g{\Lambda}_j^* =
\dot{\g{\lambda}}^I(\tau_i), \quad 1 \le i < N, \quad
\sum_{j=1}^{N} D_{Nj}^\ddagger \g{\Lambda}_j^* =
\dot{\g{\lambda}}^I(1) - \g{\lambda}^I(1)/\omega_N. \label{h83}
\end{equation}
Since $\g{\lambda}^*$ satisfies (\ref{costate}), it follows that
\begin{equation} \label{h84}
\nabla_x H({\bf X}_i^*,{\bf U}_i^*, {\g \Lambda}_i^*) =
\nabla_x H({\bf x}^*(\tau_i),{\bf u}^*(\tau_i), {\g \lambda}^*(\tau_i)) =
-\dot{\g{\lambda}}^*(\tau_i) .
\end{equation}
Hence, for $i < N$, we have
\begin{equation}\label{t4}
\C{T}_{4i}(\m{X}^*,  \m{U}^*, \g{\Lambda}^*) =
\dot{\g{\lambda}}^I(\tau_i) - \dot{\g{\lambda}}^*(\tau_i) .
\end{equation}
By a similar analysis as that used for $\C{T}_1$ in (\ref{h2}),
we conclude that
\begin{equation}\label{t4b}
|\C{T}_{4i}(\m{X}^*,  \m{U}^*, \g{\Lambda}^*)| \le c N^{2-\eta} \quad
\mbox{for } i < N.
\end{equation}
The difference between the analysis of the state in (\ref{h2}) and the
analysis of the costate in (\ref{t4b}) is that the $O(\log N)$ bound
for the Lebesgue constant of $\tau_0$, $\ldots$, $\tau_N$ must be
replaced by an $O(\sqrt{N})$ bound, derived in Theorem 5.1 of
\cite{HagerHouRao15a}, for the Lebesgue constant of
$\tau_1$, $\ldots$, $\tau_N$.
This difference between the state and the costate
arises since the state interpolant
$\m{x}^I \in \C{P}_N^n$ interpolates $\m{x}^*$ at $\tau_0$, $\ldots$, $\tau_N$
while the costate interpolant
$\g{\lambda}^I \in \C{P}_{N-1}^n$ interpolates $\g{\lambda}^*$
at $\tau_1$, $\ldots$, $\tau_N$.
Since the Lebesgue constant term in the bound for
$\|\C{T}_{4i}(\m{X}^*,  \m{U}^*, \g{\Lambda}^*)\|$ is dominated by
the other terms, the right sides of (\ref{h2}) and (\ref{t4b}) are the same.

Similarly, using (\ref{h84}) and the second equation in (\ref{h83}),
it follows that
\[
\C{T}_5 (\m{X}^*, \m{U}^*, \g{\Lambda}^*) =
\dot{\g{\lambda}}^I(1) - \dot{\g{\lambda}}^*(1) +
(\nabla C (\m{x}^*(1)) - \g{\lambda}^*(1))/\omega_N =
\dot{\g{\lambda}}^I(1) - \dot{\g{\lambda}}^*(1)
\]
since $\g{\lambda}^I(1) = \g{\lambda}^*(1) = \nabla C (\m{x}^*(1))$
by (\ref{costate}).
Hence, just like the bound for $\C{T}_{4i}$ in (\ref{t4b}), we have
\[
|\C{T}_5 (\m{X}^*, \m{U}^*, \g{\Lambda}^*)| \le cN^{2-\eta}.
\]

Now consider $\C{T}_3$.
By (\ref{h84}), the definition $\g{\Lambda}_0^*= \g{\lambda}^*(-1)$, and
the terminal condition $\g{\lambda}^*(1) = \nabla C (\m{x}^*(1))$ from
(\ref{costate}), we have
\begin{equation}\label{Tfive}
\C{T}_3(\m{X}^*, \m{U}^*, \g{\Lambda}^*)
={\g \lambda}^*(-1)-{\g \lambda}^*(1)
+\sum_{i=1}^{N}\omega_i \dot{\g{\lambda}}^* (\tau_i) .
\end{equation}
By the fundamental theorem of calculus and the fact that
$N$-point Radau quadrature is exact for polynomials of degree up to
$2N - 2$, we have
\begin{equation} \label{h3}
\m{0} =
{\g \lambda}^I(-1)-{\g \lambda}^I(1)+\int_{-1}^1\dot{\g \lambda}^I(t)dt =
{\g \lambda}^I(-1)-{\g \lambda}^I(1)+
\sum_{j=1}^N\omega_j\dot{\g \lambda}^I(\tau_j) .
\end{equation}
Subtract (\ref{h3}) from (\ref{Tfive}) to obtain
\begin{equation}\label{h4}
\C{T}_5 (\m{X}^*,  \m{U}^*, \g{\Lambda}^*) =
{\g \lambda}^*(-1)-{\g \lambda}^I(-1)
+\sum_{j=1}^N\omega_j\left(\dot{\g \lambda}^*(\tau_j)
-\dot{\g \lambda}^I(\tau_j)\right)
\end{equation}
since $\g{\lambda}^I(1) = \g{\lambda}^*(1)$.
Since $\omega_i > 0$ and their sum is 2, it follows from
(\ref{t4}) and (\ref{t4b}) that
\begin{equation}\label{h5}
\sum_{j=1}^N\omega_j\left|\dot{\g \lambda}^I(\tau_j)
-\dot{\g \lambda}^*(\tau_j)\right| \le c N^{2-\eta} .
\end{equation}
By Theorem~15.1 in \cite{Trefethen13} and
Lemma~2.2 and Theorem~5.1 in \cite{HagerHouRao15a}, we have
\begin{eqnarray}
| {\g \lambda}^*(-1)-{\g \lambda}^I(-1)| &\le&
(1+ c_1 \sqrt{N}) \left( \frac{12}{\eta+2} \right)
\left( \frac{6e}{N} \right)^{\eta+1}
\|{\g{\lambda}^*}^{(\eta+1)}\|_\infty \nonumber \\
&\le& cN^{-(0.5+\eta)}.  \label{h6}
\end{eqnarray}
We combine (\ref{h4})--(\ref{h6}) to see that
$\C{T}_5 (\m{X}^*,  \m{U}^*, \g{\Lambda}^*)$
also complies with the bound (\ref{delta}).
This completes the proof.
\end{proof}

\section{Invertibility}
\label{inverse}
In this section, we show that the derivative $\nabla \C{T} (\g{\theta}^*)$
is invertible.
This is equivalent to showing that for each $\m{y} \in \C{Y}$,
there is a unique $\g{\theta} \in \C{X}$ such that
$\nabla \C{T} (\g{\theta}^*)[\g{\theta}] = \m{y}$.
In our case, $\g{\theta}^* = (\m{X}^*, \m{U}^*, \g{\Lambda}^*)$
and $\g{\theta} = (\m{X}, \m{U}, \g{\Lambda})$.
To simplify the notation, we let $\nabla \C{T}^*[\m{X}, \m{U}, \g{\Lambda}]$
denote the derivative of $\C{T}$ evaluated at
$(\m{X}^*, \m{U}^*, \g{\Lambda}^*)$ operating on $(\m{X}, \m{U}, \g{\Lambda})$.
This derivative involves the following 6 matrices:
\[
\begin{array}{ll}
{\bf A}_i=\nabla_x{\bf f}({\bf x}^*(\tau_i),{\bf u}^*(\tau_i)),
&{\bf B}_i=\nabla_u{\bf f}({\bf x}^*(\tau_i),{\bf u}^*(\tau_i)),\\
{\bf Q}_i=\nabla_{xx}H\left({\bf x}^*(\tau_i),{\bf u}^*(\tau_i),
{\g \lambda}^*(\tau_i)\right),
&{\bf S}_i=\nabla_{xu}H\left({\bf x}^*(\tau_i),{\bf u}^*(\tau_i),
{\g \lambda}^*(\tau_i)\right),\\
{\bf R}_i=\nabla_{uu}H\left({\bf x}^*(\tau_i),{\bf u}^*(\tau_i),
{\g \lambda}^*(\tau_i)\right), &{\bf T}=\nabla^2C({\bf x}^*(1)).
\end{array}
\]
With this notation,
the 6 components of $\nabla \C{T}^*[\m{X}, \m{U}, \g{\Lambda}]$ are as follows:
\begin{eqnarray*}
\nabla \C{T}_{1i}^*[\m{X}, \m{U}, \g{\Lambda}] &=&
\left( \sum_{j=0}^{N}{D}_{ij}{\bf X}_j \right)
- \m{A}_i \m{X}_i - \m{B}_i \m{U}_i,
\quad 1\leq i\leq N ,\\
\nabla \C{T}_2^*[\m{X}, \m{U}, \g{\Lambda}] &=&
{\bf X}_{0},
\\
\nabla \C{T}_{3}^*[\m{X}, \m{U}, \g{\Lambda}] &=&
\g{\Lambda}_0 - \m{TX}_N - \sum_{i=1}^N \omega_i
(\m{A}_i\tr \g{\Lambda}_i  + \m{Q}_i \m{X}_i + \m{S}_i \m{U}_i),
\\
\nabla \C{T}_{4i}^*[\m{X}, \m{U}, \g{\Lambda}] &=&
\left( \sum_{j=1}^{N}{D}_{ij}^\ddagger{\g \Lambda}_j \right) +
\m{A}_i\tr \g{\Lambda}_i  + \m{Q}_i \m{X}_i + \m{S}_i \m{U}_i,
\quad 1 \leq i < N , \\
\nabla \C{T}_{5}^*[\m{X}, \m{U}, \g{\Lambda}] &=&
\left( \sum_{j=1}^{N}{D}_{Nj}^\ddagger{\g \Lambda}_j \right) +
\m{A}_N\tr \g{\Lambda}_N  + \m{Q}_N \m{X}_N + \m{S}_N \m{U}_N
+ \m{T} \m{X}_N/\omega_N , \\
\nabla \C{T}_{6i}^*[\m{X}, \m{U}, \g{\Lambda}] &=&
\m{S}_i\tr\m{X}_i + \m{R}_i \m{U}_i + \m{B}_i\tr \g{\Lambda}_i,
\quad 1\leq i\leq N .
\end{eqnarray*}

The analysis of invertibility starts with results concerning the invertibility
of the linearized discrete state dynamics.
\smallskip

\begin{lemma}
\label{feasiblestate}
If {\rm (P1)} and {\rm (A3)} hold,
then for each $\m{q} \in \mathbb{R}^{n}$ and $\m{p} \in \mathbb{R}^{nN}$
with $\m{p}_i \in \mathbb{R}^n$,
the linear system
\begin{eqnarray}
\left( \sum_{j=0}^{N}{D}_{ij}{\bf X}_j \right) - \m{A}_i \m{X}_i  &=& \m{p}_{i}
\quad 1\leq i\leq N , \label{h99} \\
{\bf X}_0&=&{\m q}, \label{h100}
\end{eqnarray}
has a unique solution $\m{X}_j \in \mathbb{R}^{n}$,
$0 \le j \le N$.
This solution has the bound
\begin{equation}\label{xjbound}
\|\m{X}_j\|_\infty \le 4\|\m{p}\|_\infty + 2\| \m{q}\|_\infty, \quad
0 \le j \le N.
\end{equation}
\end{lemma}
\smallskip

\begin{proof}
If $\m{X}$ is a solution of (\ref{h99})--(\ref{h100}),
then $\m{X}_0 = \m{q}$ and
$\m{X}_0$ trivially satisfies (\ref{xjbound}).
Next, focus on the remaining components of $\m{X}$.
Let $\bar{\m{X}}$ be the vector obtained by vertically stacking
$\m{X}_1$ through $\m{X}_N$,
let ${\m{A}}$ be the block diagonal matrix
with $i$-th diagonal block $\m{A}_i$, $1 \le i \le N$,
and define
$\bar{\m{D}} = {\bf D}_{1:N}\otimes {\bf I}_n$
and $\bar{\m{D}}_0 = \m{D}_0 \otimes \m{I}_n$
where $\otimes$ is the Kronecker product.
With this notation, the linear system (\ref{h99})--(\ref{h100}) reduces to
\begin{equation}\label{h300}
(\bar{\m{D}} - {\m{A}}) \bar{\m{X}} = \m{p} - \bar{\m{D}}_0 \m{q}.
\end{equation}
By (P1), ${\bf D}_{1:N}$ is invertible which implies that
$\bar{\m{D}}$ is invertible and
$\bar{\m{D}}^{-1} = {\bf D}_{1:N}^{-1}\otimes {\bf I}_n$.
Since the polynomial that is identically equal to 1 has derivative 0 and
since $\m{D}$ is a differentiation matrix, we have $\m{D1} = \m{0}$,
which implies that $\m{D}_{1:N}^{-1} \m{D}_0 = -\m{1}$.
It follows that
\[
\bar{\m{D}}^{-1}
\bar{\m{D}}_{0} =
[{\bf D}_{1:N}^{-1}\otimes {\bf I}_n]
[\m{D}_{0} \otimes \m{I}_n] = -\m{1}\otimes \m{I}_n .
\]

Multiply (\ref{h300}) by $\bar{\m{D}}^{-1}$ to obtain
\begin{equation}\label{h301}
(\m{I} - \bar{\m{D}}^{-1} \m{A}) \bar{\m{X}} =
\bar{\m{D}}^{-1} \m{p} + (\m{1}\otimes \m{I}_n) \m{q}.
\end{equation}
By (P1) $\|\m{D}_{1:N}^{-1}\|_\infty \le 2$, which implies that
$\|\bar{{\bf D}}^{-1}\|_\infty \le 2$.
By (A3) $\|{\m{A}}\|_\infty \le 1/4$.
By \cite[p. 351]{HJ12},
${\bf I}-\bar{\bf D}^{-1}{\bf A}$ is invertible and
$\left\|({\bf I}-\bar{\bf D}^{-1}{\bf A})^{-1}\right\|_\infty \leq 2$.
Multiply (\ref{h301}) by $({\bf I}-\bar{\bf D}^{-1}{\bf A})^{-1}$ and
take the norm of each side to obtain
$\|\bar{\m{X}}\|_\infty \le 4 \|\m{p}\|_\infty + 2\|\m{q}\|_\infty$.
This complete the proof of (\ref{xjbound}).
\end{proof}
\smallskip


Next, we establish the invertibility of $\nabla \C{T}^*$.
\smallskip

\begin{proposition}
If {\rm (P1)}, {\rm (A2)}, and {\rm (A3)} hold, then
$\nabla \C{T}^*$ is invertible.
\end{proposition}
\smallskip

\begin{proof}
Our approach is to formulate a strongly convex quadratic programming
problem which has a unique solution $(\m{X}, \m{U})$ by (A2),
and which has the property that the associated
first-order optimality condition is
$\nabla \C{T}^*[\m{X}, \m{U}, \g{\Lambda}] = \m{y}$.
Since $\nabla \C{T}^*$ is square and
$\nabla \C{T}^*[\m{X}, \m{U}, \g{\Lambda}] = \m{y}$ has a solution
for each choice of $\m{y}$, we conclude the $\nabla \C{T}^*$ is invertible.

The quadratic program is
\begin{equation}\label{QP}
\left.
\begin{array}{cl}
\mbox {minimize} &\frac{1}{2} \mathcal{Q}({\bf X},{\bf U})
+ \C{L}(\m{X}, \m{U}) \\[.08in]
\mbox {subject to} &\sum_{j=0}^{N}{D}_{ij}{\bf X}_j
={\bf A}_i{\bf X}_i+ {\bf B}_i{\bf U}_i+{\bf y}_{1i},
\quad 1\leq i \leq N, \\
&{\bf X}_{0}={\bf y}_2 ,
\end{array}
\right\}
\end{equation}
where the quadratic and linear terms in the objective are
\begin{eqnarray}
\mathcal{Q}({\bf X},{\bf U})&=&
{\bf X}_{N}\tr{\bf T}{\bf X}_{N}+\sum_{i=1}^N\omega_i
\left({\bf X}_i\tr{\bf Q}_i{\bf X}_i
+2{\bf X}_i\tr{\bf S}_i{\bf U}_i
+{\bf U}_i\tr{\bf R}_i{\bf U}_i\right), \label{Q}\\
\C{L}(\m{X}, \m{U}) &=&
\m{X}_0\tr \left( \m{y}_3 + \sum_{i=1}^N \omega_i \m{y}_{4i} \right) -
\sum_{i=1}^N \omega_i \left( \m{y}_{4i}\tr \m{X}_i +
\m{y}_{6i}\tr \m{U}_i \right) . \label{L}
\end{eqnarray}
In (\ref{L}) we simplified the formula for $\C{L}$ by introducing
$\m{y}_{4N} = \m{y}_5$.
By Lemma~\ref{feasiblestate}, the quadratic programming problem (\ref{QP})
is feasible.
Since the Radau quadrature weights $\omega_i$ are strictly positive,
it follows from (A2) that $\C{Q}$ is strongly convex.
Hence, there exists a unique optimal solution to (\ref{QP})
for any choice of $\m{y}$.
Since the constraints are linear, the first-order optimality conditions hold.
The linear term was chosen so that the first-order optimality
conditions for (\ref{QP}) reduce to 
$\nabla \C{T}^*[\m{X}, \m{U}, \g{\Lambda}] = \m{y}$.
Hence, the existence of a solution to (\ref{QP}) for each $\m{y}$
would imply that $\nabla \C{T}^*[\m{X}, \m{U}, \g{\Lambda}] = \m{y}$
has a solution for each choice of $\m{y}$.
To complete the proof, we need to show that the first-order
optimality conditions for (\ref{QP}) are equivalent to
$\nabla \C{T}^*[\m{X}, \m{U}, \g{\Lambda}] = \m{y}$.

The Lagrangian of (\ref{QP}) is given by
\[
\frac{1}{2} \C{Q}(\m{X}, \m{U}) + \C{L}(\m{X}, \m{U})
+\sum_{i=1}^N \g{\lambda}_i\tr \left(
{\bf A}_i{\bf X}_i+ {\bf B}_i{\bf U}_i+{\bf y}_{1i}
- \sum_{j=0}^{N}{D}_{ij}{\bf X}_j \right)
+ \g{\Lambda}_0\tr (\m{y}_2 - \m{X}_0) .
\]
The first-order optimality conditions are obtained by setting to zero
the derivative of the Lagrangian with respect to each of the
components of $\m{X}$ and $\m{U}$.
We give the derivation of the 3rd, 4th, and 5th components of
$\nabla \C{T}^*[\m{X}, \m{U}, \g{\Lambda}] = \m{y}$, while the 6th
component follows in a similar fashion,
and the 1st and 2nd components are simply the constraints in (\ref{QP}).

Setting to zero the partial derivative of the Lagrangian with
respect to $\m{X}_i$, $1 \le i < N$, yields the equation
\[
\m{A}_i\tr \g{\lambda}_i +\omega_i \m{Q}_i \m{X}_i
+ \omega_i \m{S}_i \m{U}_i - \sum_{j=1}^N D_{ji} \g{\lambda}_j =
\omega_i \m{y}_{4i} .
\]
Substituting $D_{ji} = -(\omega_i/\omega_j) D_{ij}^\ddag$ and
$\g{\lambda}_j = \omega_j \g{\Lambda}_j$, we obtain
\begin{equation}\label{4th}
\left( \sum_{j=1}^N D_{ij}^\ddag \g{\Lambda}_j \right) +
\m{A}_i\tr \g{\Lambda}_i +\m{Q}_i \m{X}_i + \m{S}_i \m{U}_i
= \m{y}_{4i} ,
\end{equation}
which gives the 4th component of
$\nabla \C{T}^*[\m{X}, \m{U}, \g{\Lambda}] = \m{y}$.
In a similar fashion, setting to
zero the partial derivative of the Lagrangian with
respect to $\m{X}_N$ yields the 5th component of
$\nabla \C{T}^*[\m{X}, \m{U}, \g{\Lambda}] = \m{y}$:
\begin{equation}\label{5th}
\left( \sum_{j=1}^N D_{Nj}^\ddag \g{\Lambda}_j \right) +
\m{A}_N\tr \g{\Lambda}_N +\m{Q}_N \m{X}_N + \m{S}_N \m{U}_N +
\m{T}\m{X}_N/\omega_N = \m{y}_{4N} .
\end{equation}

Setting to zero the partial derivative of the Lagrangian with respect
to $\m{X}_0$ gives the equation
\begin{equation}\label{3rd}
\g{\Lambda}_0 + \sum_{i=1}^N
\left( \g{\lambda}_i D_{i0} - \omega_i \m{y}_{4i} \right) = \m{y}_3 .
\end{equation}
Since $\m{D}$ is a differentiation matrix and $\m{D1} = \m{0}$,
it follows that
\[
D_{i0} = -\sum_{j=1}^N D_{ij}.
\]
Consequently, we have
\[
\sum_{i=1}^N D_{i0} \g{\lambda}_i =
- \sum_{i=1}^N \sum_{j=1}^N D_{ij} \g{\lambda}_i =
\sum_{i=1}^N \sum_{j=1}^N \omega_j D_{ji}^\ddag \g{\Lambda}_i =
\sum_{i=1}^N \sum_{j=1}^N \omega_i D_{ij}^\ddag \g{\Lambda}_j .
\]
We make this substitution as well as (\ref{4th}) and (\ref{5th}) into
(\ref{3rd}).
The $\m{D}^\ddag$ terms cancel to give
\[
\g{\Lambda}_0 - \m{TX}_N - \sum_{i=1}^N \omega_i
(\m{A}_i\tr \g{\Lambda}_i  + \m{Q}_i \m{X}_i + \m{S}_i \m{U}_i) = \m{y}_3,
\]
which is the 3rd component of
$\nabla \C{T}^*[\m{X}, \m{U}, \g{\Lambda}] = \m{y}$.
This completes the proof.
\end{proof}

\section{$\omega$-norm bounds for the state and control}
\label{omegabounds}
In this section we obtain a bound for the
$(\m{X}, \m{U})$ component of the solution to
$\nabla \C{T}^*[\m{X}, \m{U}, \g{\Lambda}] = \m{y}$ in terms of $\m{y}$.
Since $\m{X}_0$ must satisfy the constraint $\m{X}_0 = \m{y}_2$,
it is trivially bounded in terms of $\|\m{y}\|_\infty$.
Hence, we focus on $(\m{X}_i, \m{U}_i)$, $1 \le i \le N$.
The bound we derive in this section is in terms of the
$\omega$-norms defined by
\begin{equation}\label{onorm}
\|{\bf X}\|_\omega^2= |\m{X}_N|^2 +
\sum_{i=1}^N\omega_i|{\bf X}_i|^2 \quad \mbox{and} \quad
\|{\bf U}\|_\omega^2= \sum_{i=1}^N\omega_i|{\bf U}_i|^2.
\end{equation}
This defines a norm since the Radau quadrature weight $\omega_i > 0$
for each $i$.
Since the
$(\m{X}, \m{U})$ component of the solution to
$\nabla \C{T}^*[\m{X}, \m{U}, \g{\Lambda}] = \m{y}$ is a solution
of the quadratic program (\ref{QP}), we will bound the solution
to the quadratic program.

First, let us think more abstractly.
Let $\pi$ be a symmetric, continuous bilinear functional defined on
a Hilbert space $\C{H}$, let $\ell$ be a continuous linear functional,
let $\phi \in \C{H}$, and consider the quadratic program
\[
\min \left\{ \frac{1}{2}\pi(v+\phi,v+\phi) + \ell (v+\phi):
v\in \C{V} \right\},
\]
where $\C{V}$ is a subspace of $\C{H}$.
If $w$ is a minimizer, then by the first-order optimality conditions,
we have
\[
\pi (w, v) + \pi (\phi, v) + \ell (v) = 0 \quad \mbox{for all } v \in \C{V}.
\]
Inserting $v = w$ yields
\begin{equation}\label{nullw}
\pi (w,w) = -(\pi (w, \phi) + \ell (w)).
\end{equation}

We apply this observation to the quadratic program (\ref{QP}) where we
treat $\m{X}_0 = \m{y}_2$ as fixed, so the minimization is over
$(\m{X}_i, \m{U}_i)$, $1 \le i \le N$.
We identify $\ell$ with the linear functional $\C{L}$ in (\ref{L}) but with
the $\m{X}_0$ term dropped since it is fixed,
and $\pi$ with the bilinear form associated with the quadratic term (\ref{Q}).
The subspace $\C{V}$ is the null space of the linear operator
in (\ref{QP}) and $\phi$ is a particular solution of the linear system.
The complete solution of (\ref{QP}) is the particular solution plus
the minimizer over the null space.

In more detail, let $\g{\chi}$ denote the solution to (\ref{h99})--(\ref{h100})
given by Lemma~\ref{feasiblestate} for $\m{p} = \m{y}_1$ and $\m{q} = \m{y}_2$.
We consider the particular solution $(\m{X}, \m{U})$
of the linear system in (\ref{QP}) given by $(\g{\chi},\m{0})$.
The relation (\ref{nullw}) describing the null space component $(\m{X}, \m{U})$
of the solution is
\begin{equation}\label{nullx}
\C{Q}(\m{X}, \m{U}) = - \left( \g{\chi}_N\tr \m{TX}_N
+ \sum_{i=1}^N \omega_i \left[ (\m{Q}_i \g{\chi}_i - \m{y}_{4i})\tr \m{X}_i -
\m{y}_{6i}\tr \m{U}_i \right] \right) .
\end{equation}
Here the terms containing $\g{\chi}$ are associated with $\pi (w, \phi)$,
while the remaining terms are associated with $\ell$,
or equivalently with $\C{L}$.
By (A2) we have the lower bound
\begin{equation}\label{lower}
\C{Q}(\m{X}, \m{U}) \ge \alpha (\|\m{X}\|_\omega^2 + \|\m{U}\|_\omega^2) .
\end{equation}
All the terms on the right side of (\ref{nullx}) can be bounded with
the Schwarz inequality; for example,
\begin{eqnarray}
\sum_{i=1}^N \omega_i \m{y}_{4i}\tr \m{X}_i &\le&
\left( \sum_{i=1}^N \omega_i |\m{y}_{4i}|^2 \right)^{1/2}
\left( \sum_{i=1}^N \omega_i |\m{X}_{i}|^2 \right)^{1/2} \nonumber \\
&\le& \sqrt{2} \|\m{y}_4\|_\infty
\left( \|\m{X}\|_\omega^2 + \|\m{U}\|_\omega^2  \right)^{1/2} . \label{upper}
\end{eqnarray}
The last inequality exploits the fact that the $\omega_i$ sum to 2 and
$|\m{y}_{4i}| \le \|\m{y}_4\|_\infty$.
To handle the terms involving $\g{\chi}$ in (\ref{nullx}),
we utilize the upper bound
$\|\g{\chi}_j\|_\infty \le 6 \|\m{y}\|_\infty$ based on Lemma~\ref{feasiblestate}
with $\m{p} = \m{y}_1$ and $\m{q} = \m{y}_2$.
Combining upper bounds of the form (\ref{upper}) with the lower bound
(\ref{lower}), we conclude from (\ref{nullx}) that both $\|\m{X}\|_\omega$ and
$\|\m{U}\|_\omega$ are bounded by a constant times $\|\m{y}\|_\infty$.
The complete solution of (\ref{QP}) is the null space component that we
just estimated plus the particular solution $(\g{\chi}, \m{0})$.
Again, since $\|\g{\chi}_j\|_\infty \le 6 \|\m{y}\|_\infty$,
we obtain the following result.
\smallskip

\begin{lemma}\label{l2}
If {\rm (A2)--(A3)} and {\rm (P1)} hold, then there exists a constant $c$,
independent of $N$, such that the solution $(\m{X}, \m{U})$
of $(\ref{QP})$ satisfies $\|\m{X}\|_\omega \le c \|\m{y}\|_\infty$ and
$\|\m{U}\|_\omega \le c \|\m{y}\|_\infty$.
\end{lemma}
\smallskip

\section{$\infty$-norm bounds}
\label{inftybounds}
We now need to convert these $\omega$-norm bounds for $\m{X}$ and
$\m{U}$ into $\infty$-norm bounds and at the same time, obtain an
$\infty$-norm bound for $\g{\Lambda}$.
As in Lemma~\ref{feasiblestate}, the
solution to the dynamics in (\ref{QP})
can be expressed
\begin{equation}\label{solvestate}
\bar{\m{X}} = (\m{I} - \bar{\m{D}}^{-1}{\m{A}})^{-1}
\left[ \bar{\m{D}}^{-1} \m{BU} + \bar{\m{D}}^{-1}\m{y}_1
+ (\m{1}\otimes \m{I}_n) \m{y}_2 \right],
\end{equation}
where $\m{B}$ is the block diagonal matrix with $i$-th diagonal block $\m{B}_i$
and $\m{U}$ is obtained by vertically stacking $\m{U}_1$ through $\m{U}_N$.
By Lemma~\ref{feasiblestate},
\begin{equation}\label{h201}
\|(\m{I} - \bar{\m{D}}^{-1}{\m{A}})^{-1}
\left[ \bar{\m{D}}^{-1}\m{y}_1
+ (\m{1}\otimes \m{I}_n) \m{y}_2 \right]\le
4 \|\m{y}_1\|_\infty + 2 \|\m{y}_2\|_\infty .
\end{equation}
The term $\bar{\m{D}}^{-1} \m{BU}$ can be bounded using (P2) and the
strategy given in Section~6 of \cite{HagerHouRao15b}.
That is, we first observe that
\begin{equation}\label{h103}
\bar{\m{D}}^{-1} \m{BU} = [\m{D}_{1:N}^{-1}\otimes \m{I}_n] \m{BU} =
[(\m{W}^{1/2} \m{D}_{1:N})^{-1} \otimes \m{I}_n] \m{BU}_\omega ,
\end{equation}
where $\m{W}$ is the diagonal matrix with the quadrature weights on the
diagonal and $\m{U}_\omega$ is the vector whose $i$-th element is
$\sqrt{\omega_i} \m{U}_i$;
the $\sqrt{\omega_i}$ factors in (\ref{h103}) cancel each other.
An element of the vector $\bar{\m{D}}^{-1} \m{BU}$ is the
dot product between a row of
$(\m{W}^{1/2} \m{D}_{1:N})^{-1} \otimes \m{I}_n$ and the column vector
$\m{BU}_\omega$.
By (P2) the rows of $(\m{W}^{1/2} \m{D}_{1:N})^{-1} \otimes \m{I}_n$ have
Euclidean length bounded by $\sqrt{2}$.
By the properties of matrix norms induced by vector norms, we have
\[
\|\m{BU}_\omega\|_2 \le \|\m{B}\|_2 \|\m{U}_\omega\|_2 =
\|\m{B}\|_2 \|\m{U}\|_\omega .
\]
It follows that
\[
\|\bar{\m{D}}^{-1} \m{BU}\|_\infty \le \sqrt{2} \|\m{B}\|_2 \|\m{U}\|_\omega \le
c \|\m{y}\|_\infty,
\]
where the generic constant $c$ is independent of $N$ by Lemma~\ref{l2}.
Combining this with (\ref{h201}), we conclude that
for some constant $c$, the $\bar{\m{X}}$ component of the solution to
(\ref{QP}) satisfies $\|\bar{\m{X}}\|_\infty \le c\|\m{y}\|_\infty$.
Since $|\m{X}_0| = |\m{y}_2| \le \|\m{y}\|_\infty$,
the entire $\m{X}$-component of the solution to (\ref{QP}) satisfies
$\|\m{X}\|_\infty \le c\|\m{y}\|_\infty$ for some $c$.

Next, we focus on the 4th and 5th components of
$\nabla \C{T}^*[\m{X}, \m{U}, \g{\Lambda}] = \m{y}$ which can be written
\begin{equation}\label{c1}
\bar{\m{D}}^\ddag \bar{\g{\Lambda}} + {\m{A}}\tr \bar{\g{\Lambda}} +
{\m{Q}} \bar{\m{X}} + {\m{S}} {\m{U}}
+\left( \frac{1}{\omega_N} \right) (\m{e}_N\otimes \m{I}_n) \m{TX}_N = \m{y}_4,
\end{equation}
where $\bar{\m{D}}^\ddag = {\bf D}^\ddag\otimes {\bf I}_n$,
$\bar{\g{\Lambda}}$ is obtained by vertically stacking
$\g{\Lambda}_1$ through $\g{\Lambda}_N$,
${\m{Q}}$ and ${\m{S}}$ are block diagonal matrices with $i$-th diagonal blocks
$\m{Q}_i$ and $\m{S}_i$ respectively, and $\m{e}_N \in \mathbb{R}^N$
is the vector whose components are all zero except for the $N$-th
component which is 1.
Similar to our manipulations of the state dynamics in (\ref{solvestate}),
we use (A3) and (P3) to solve for $\bar{\g{\Lambda}}$:
\begin{equation}\label{h204}
\bar{\g{\Lambda}} =
-(\m{I} + \bar{\m{D}}^{\ddag\; -1}{\m{A}}\tr)^{-1} \bar{\m{D}}^{\ddag\;-1}
\left[ \m{SU} + \m{Q}\bar{\m{X}}
+\left( \frac{1}{\omega_N} \right) (\m{e}_N\otimes \m{I}_n)
\m{TX}_N - \m{y}_4 \right]
\end{equation}

Let $p$ be the polynomial that is identically one, and let $\m{p}$
be the associated vector whose components are all one.
Making this substitution in (\ref{h282}) gives
${\bf D}^\ddag{\bf 1}=$ $-\m{e}_N /\omega_N$,
which implies that
\begin{equation}\label{h999}
\m{D}^{\ddag\; -1} \m{e}_N = -\omega_N \m{1}.
\end{equation}
Consequently, we have
\[
\bar{\m{D}}^{\ddag\;-1} (\m{e}_N\otimes \m{I}_n)/\omega_N =
[{\bf D}^{\ddag \; -1}\otimes {\bf I}_n]
(\m{e}_N\otimes \m{I}_n)/\omega_N =
-\m{1}\otimes \m{I}_n .
\]
With this substitution, (\ref{h204}) becomes
\begin{equation}\label{h101}
\bar{\g{\Lambda}} =
-(\m{I} + \bar{\m{D}}^{\ddag\; -1}{\m{A}}\tr)^{-1}
\left[
\bar{\m{D}}^{\ddag\;-1} (\m{SU} + \m{Q}\bar{\m{X}}
- \m{y}_4 ) -
(\m{1}\otimes \m{I}_n) \m{TX}_N \right] .
\end{equation}
Exactly as in Lemma~\ref{feasiblestate}, we have
$\|\bar{\m{D}}^{\ddag\; -1}\|_\infty \le 2$ by (P2), and
$(\m{I} + \bar{\m{D}}^{\ddag\; -1}{\m{A}}\tr)^{-1}\|_\infty \le 2$ by (A3).
As in the analysis of the state dynamics, all the terms on the right side
of (\ref{h101}) are bounded by $c\|\m{y}\|_\infty$ for a suitable choice of $c$.
Hence, we have $\|\bar{\g{\Lambda}}\|_\infty \le c\|\m{y}\|_\infty$.

Now consider the 6th component of
$\nabla \C{T}^*[\m{X}, \m{U}, \g{\Lambda}] = \m{y}$, which can be written
\[
\m{S}_i\tr\m{X}_i + \m{R}_i \m{U}_i + \m{B}_i\tr \g{\Lambda}_i = \m{y}_{6i},
\quad 1\leq i\leq N .
\]
Previously, we have shown the existence of a constant $c$, independent of $N$,
such that $\|\m{X}_i\|_\infty \le c \|\m{y}\|_\infty$ and
$\|\g{\Lambda}_i\|_\infty \le c \|\m{y}\|_\infty$, $1 \le i \le N$.
By (A2) the smallest eigenvalue of $\m{R}_i$ is bounded from below by $\alpha$.
Hence, there also exists a constant $c$ such that
$\|\m{U}_i \|_\infty \le c \|\m{y}\|_\infty$.
Finally, the 3rd component of
$\nabla \C{T}^*[\m{X}, \m{U}, \g{\Lambda}] = \m{y}$ can be written
\[
\g{\Lambda}_0 - \m{TX}_N - \sum_{i=1}^N \omega_i
(\m{A}_i\tr \g{\Lambda}_i  + \m{Q}_i \m{X}_i + \m{S}_i \m{U}_i) = \m{y}_3.
\]
By the uniform bounds on $\|\m{X}_i\|_\infty$,
$\|\m{U}_i\|_\infty$, and $\|\g{\Lambda}_i\|_\infty$, $1 \le i \le N$,
there also exists a constant $c$, independent of $N$, such that
$\|\g{\Lambda}_0\|_\infty \le c\|\m{y}\|_\infty$.
We summarize these results as follows:

\begin{lemma}\label{inf-bounds}
If {\rm (A2)--(A3)} and {\rm (P1)--(P4)} hold,
then there exists a constant $c$,
independent of $N$, such that the solution
of $\nabla \C{T}^*[\m{X}, \m{U}, \g{\Lambda}] = \m{y}$ satisfies
\[
\|\m{X}\|_\infty + \|\m{U}\|_\infty + \|\g{\Lambda}\|_\infty \le
c \|\m{y}\|_\infty.
\]
\end{lemma}

The proof of Theorem~\ref{maintheorem} follows exactly as in
\cite{HagerHouRao15b}.
By Lemma~\ref{inf-bounds},
$\mu = \|\nabla \C{T}(\m{X}^*, \m{U}^*, \g{\Lambda}^*)^{-1}\|_\infty$
is bounded uniformly in $N$.
Choose $\varepsilon$ small enough that $\varepsilon\mu < 1$.
When we compute the difference
$\nabla \C{T}(\m{X}, \m{U}, \g{\Lambda}) -
\nabla \C{T}(\m{X}^*, \m{U}^*, \g{\Lambda}^*)$ for
$(\m{X}, \m{U}, \g{\Lambda})$ near $(\m{X}^*, \m{U}^*, \g{\Lambda}^*)$
in the $\infty$-norm,
the $\m{D}$ and $\m{D}^\ddag$ constant terms cancel, and we are
left with terms involving the difference of
derivatives of $\m{f}$ or $C$ up to second order at nearby points.
By assumption, these second derivatives are uniformly continuous on
the closure of $\Omega$ and on a ball around $\m{x}^*(1)$.
Hence, for $r$ sufficiently small, we have
\[
\
\|\nabla\mathcal{T}(\m{X}, \m{U}, \g{\Lambda})-
\nabla\mathcal{T}(\m{X}^*, \m{U}^*, \g{\Lambda}^*)\|_\infty
\leq \varepsilon
\]
whenever
\begin{equation}\label{r-bound}
\max\{\|\m{X} -\m{X}^*\|_\infty, \|\m{U} -\m{U}^*\|_\infty,
\|\g{\Lambda} - \g{\Lambda}^*\|_\infty\} \le r.
\end{equation}
By Lemma~\ref{residuallemma}, it follows that
$\left\|\mathcal{T}\left(\m{X}^*, \m{U}^*, \g{\Lambda}^*\right)\right\|
\leq(1-\mu\varepsilon)r/\mu$
for all $N$ sufficiently large.
Hence, by Proposition~\ref{prop}, there exists a solution to
$\C{T}(\m{X},\m{U}, \g{\Lambda}) = \m{0}$ satisfying (\ref{r-bound}).
Moreover, by (\ref{abs}) and (\ref{delta}), the estimate (\ref{maineq}) holds.
To complete the proof, we need to show that
$(\m{X}, \m{U})$ is a local minimizer for (\ref{nlp}).
After replacing the KKT multipliers by the transformed quantities given
by (\ref{eq12}), the Hessian of the Lagrangian is a block
diagonal matrix with the following matrices forming the diagonal blocks:
\[
\begin{array}{ll}
\omega_i \nabla_{(x,u)}^2 H(\m{X}_i, \m{U}_i, \g{\Lambda}_i), &
1 \le i < N, \\[.05in]
\omega_i \nabla_{(x,u)}^2 H(\m{X}_i, \m{U}_i, \g{\Lambda}_i) +
\nabla_{(x,u)}^2 C(\m{X}_{i}), & i = N,
\end{array}
\]
where $H$ is the Hamiltonian.
In computing the Hessian, we assume that the $\m{X}$ and $\m{U}$
variables are arranged in the following order:
$\m{X}_1$, $\m{U}_1$, $\m{X}_2$, $\m{U}_2$, $\ldots$, $\m{X}_N$, $\m{U}_N$.
By (A2) the Hessian is positive definite when evaluated at
$(\m{X}^*, \m{U}^*, \g{\Lambda}^*)$.
By continuity of the second derivative of $C$ and $\m{f}$, and by the
convergence result (\ref{maineq}), we conclude that the Hessian of the
Lagrangian, evaluated at the solution of
$\C{T}(\m{X}, \m{U}, \g{\Lambda}) = \m{0}$ satisfying (\ref{r-bound}),
is positive definite for $N$ sufficiently large.
Hence, by the second-order sufficient optimality condition
\cite[Thm. 12.6]{NocedalWright2006}, $(\m{X},\m{U})$ is a strict
local minimizer of (\ref{nlp}).
This completes the proof of Theorem~\ref{maintheorem}.

\section{Numerical illustration}
\label{numerical}
Let us consider the unconstrained control problem previously introduced in
\cite{HagerHouRao15b}:
\begin{equation}\label{exprob}
\min \left\{ -x(2) :
\dot{x}(t) = \textstyle{\frac{5}{2}}(-x(t) + x(t)u(t) - u(t)^2),
\; x(0) = 1 \right\}
\end{equation}
The optimal solution and associated costate are
\begin{eqnarray*}
x^*(t) &=& 4/a(t), \quad a(t) = 1 + 3 \exp (2.5t), \\
u^*(t) &=& x^*(t)/2, \\
\lambda^* (t) &=& -a^2(t) \exp (-2.5t)/[\exp (-5) + 9 \exp(5) + 6].
\end{eqnarray*}
Figure~\ref{example} plots the logarithm of the sup-norm error in
the state, control, and costate as a function of the number of collocation
points.
Since these plots are nearly linear, the error behaves like
$c 10^{-\alpha N}$ where $\alpha \approx 0.6$ for either the state
or the control and $\alpha \approx 0.8$ for the costate.
In Theorem~\ref{maintheorem}, the dependence of the error on $N$
is somewhat complex due to the connection between $\eta$ and $N$.
As we increase $N$, we can also increase $\eta$ when the solution is
infinitely differentiable, however, the norm of the derivatives
also enters into the error bound as in (\ref{h2}).
Nonetheless, in cases where the solution derivatives can be bounded by
$c^\eta$ for some constant $c$, it is possible to deduce an exponential
decay rate for the error as observed in \cite[Sect. 2]{GargHagerRao10a}.
\begin{figure}
\begin{center}
\includegraphics[scale=.5]{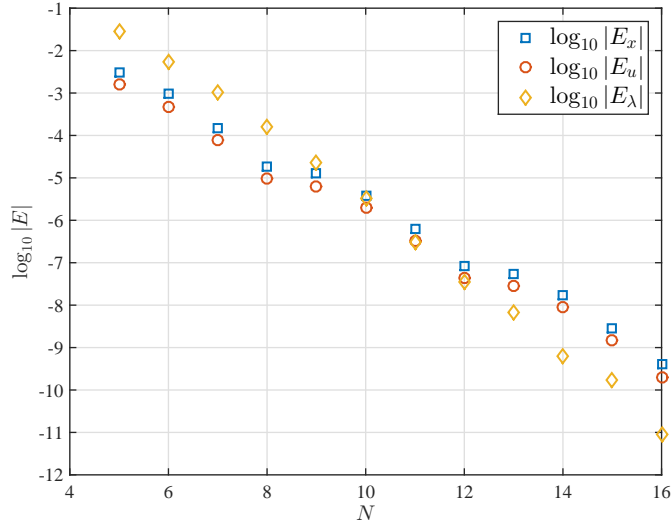}
\caption{The base 10 logarithm of the error in the sup-norm as a function of
the number of collocation points.
\label{example}}
\end{center}
\end{figure}
\section{Conclusions}
A Radau collocation scheme is analyzed for an unconstrained control problem.
For a problem with a smooth solution and a Hamiltonian
which satisfies a strong convexity assumption at a local minimizer of
the continuous problem, we show that the discrete approximation has a
local minimizer in a neighborhood of the continuous solution,
and as the number of collocation points increases, the distance in the
sup-norm between the discrete solution and the continuous solution is
$O(N^{2-\eta})$ when the continuous solution has $\eta+1$
continuous derivatives, $\eta \ge 3$,
and the number of collocation points $N$ is sufficiently large.
A numerical example is given which exhibits an exponential convergence rate.
\section{Appendix}
\label{appendix}
Before stating property (P3) in the Introduction,
we showed that ${\bf D}^\ddag$ is an invertible matrix.
In this section, we give an analytic formula for the inverse.
\begin{proposition}\label{deriv_exact}
The inverse of ${\bf D}^\ddag$ is given by
%
\[
\begin{array}{llll}
D^{\ddag \; -1}_{ij} &=& \omega_N M_j(1) + \int_1^{\tau_i} M_j (\tau) d\tau, &
1 \le i < N, \;\; 1 \le j < N, \\[.05in]
D^{\ddag \; -1}_{iN} &=& -\omega_N , & 1 \le i \le N, \\[.05in]
D^{\ddag \; -1}_{Nj} &=& \omega_N M_j (1) , & 1 \le j < N,
\end{array}
\]
where $M_j$, $1 \le j < N$, is the Lagrange interpolating basis
relative to the point set $\tau_1$, $\ldots$, $\tau_{N-1}$.
That is,
\[
M_j (\tau) =
\displaystyle\prod_{\substack{i=1\\ i\neq j}}^{N-1}
\frac{\tau-\tau_i}{\tau_j-\tau_i},\quad j=1,\ldots, N-1.
\]
\end{proposition}
\smallskip
\begin{proof}
The relation (\ref{h282}) holds for any polynomial $p$ of degree at
most $N-1$.
Let $\dot{\m{p}} \in \mathbb{R}^N$ denote the vector with
$i$-th component $\dot{p}(\tau_i)$.
In vector form, the system of equations (\ref{h282}) can be expressed
${\bf D}^\ddag{\bf p}=$ $\dot{\bf p} - \m{e}_N p(1)/\omega_N$.
Multiply by
$\m{D}^{\ddag \; -1}$ and exploit the identity
$\m{D}^{\ddag \; -1} \m{e}_N = -\omega_N \m{1}$ of (\ref{h999})
to obtain
\begin{equation}\label{h998}
{\bf D}^{\ddag \; -1} \dot{\m{p}} = \m{p} - \m{1}p(1) .
\end{equation}

Since $\dot{p}$ is a polynomial of degree at most $N-2$,
we can only specify the derivative of $p$ at $N-1$ distinct points.
Given any $j$ satisfying $1 \le j < N$, let us insert in (\ref{h998})
a polynomial $p \in \C{P}_{N-1}$ satisfying
\[
\dot{p} (\tau_j) = 1 \quad \mbox{and} \quad
\dot{p}(\tau_i) = 0
\mbox{ for all } i < N, \; \; i \ne j.
\]
A specific polynomial with this property is
\begin{equation}\label{h996}
p(\tau) = \int_{1}^{\tau} M_j(\tau) d\tau.
\end{equation}
Since $p_N = p(1) = 0$, the last component of the right side of
(\ref{h998}) vanishes to give the relation
$D_{Nj}^{\ddag \; -1} + D_{NN}^{\ddag \; -1} \dot{p}(1) = 0$.
In (\ref{h999}) we showed that all the elements in the last column of
$\m{D}^{\ddag \; -1}$ are equal to $-\omega_N$, and
by (\ref{h996}), $\dot{p}(1) = M_j (1)$.
Hence, we obtain the relation
\begin{equation}\label{h997}
D_{Nj}^{\ddag \; -1} =
-D_{NN}^{\ddag \; -1} \dot{p}(1) =
\omega_N \dot{p}(1) =
\omega_N M_j (1), \;\; 1 \le j < N.
\end{equation}

Finally, let us consider $D_{ij}^{\ddag \; -1}$ for $i < N$ and $j < N$.
We combine the $i$-th component of (\ref{h998}) for $i < N$ with
(\ref{h996}) to obtain
\begin{equation}\label{h995}
({\bf D}^{\ddag \; -1} \dot{\m{p}})_i = \int_{1}^{\tau_i} M_j (\tau) d\tau .
\end{equation}
Recall that all components of $\dot{\m{p}}$ vanish except for the $j$-th,
which is 1, and the $N$-th, which is $M_j(1)$ by (\ref{h996}).
Hence, (\ref{h995}) and the fact that
the elements in the last column of
$\m{D}^{\ddag \; -1}$ are all $-\omega_N$ yield
\[
D_{ij}^{\ddag \; -1} = \int_{1}^{\tau_i} M_j (\tau) d\tau -
D_{iN}^{\ddag \; -1} M_j(1) =
\omega_N M_j (1) + \int_{1}^{\tau_i} M_j (\tau) d\tau
\]
This completes the proof.
\end{proof}

As noted in the Introduction, the elements in the
last row of $\m{D}_{1:N}^{-1}$ are positive and sum to 2,
and the last row of
$\m{D}_{1:N}^{-1}\m{W}^{-1/2}$ is positive and has Euclidean
norm $\sqrt{2}$.
Numerically, we observe that the absolute row sums for the first
$N-1$ rows of $\m{D}_{1:N}^{-1}$ are always less than 2, while
the Euclidean norm of the first $N-1$ rows of
$\m{D}_{1:N}^{-1}\m{W}^{-1/2}$ are always less than $\sqrt{2}$.
Properties (P3) and (P4) are less tight in the sense that the
bounds only hold as equalities asymptotically.
Tables~\ref{P3} and \ref{P4} show $\|\m{D}^{\ddag \; -1}\|_\infty$
and the maximum Euclidean norm of the rows of
$\m{D}^{\ddag \; -1}\m{W}^{-1/2}$ for an increasing sequence of dimensions.
\begin{table}[ht]
\centering
\begin{tabular}{c|c|c|c|c|c|c}
\hline\hline
$N$ & 25 & 50 & 75 & 100 & 125 & 150 \\
\hline
norm & 1.995376 & 1.998844 & 1.999486 & 1.999711 & 1.999815 & 1.999871  \\
\hline\hline
$N$ & 175 & 200 & 225 & 250 & 275 & 300 \\
\hline
norm & 1.999906 & 1.999928 & 1.999943 & 1.999954 & 1.999962 & 1.999968\\
\hline
\end{tabular}
\vspace*{.1in}
\caption{$\|{\bf D}^{\ddag\; -1}\|_\infty$
\label{P3}}
\end{table}
\begin{table}[ht]
\centering
\begin{tabular}{c|c|c|c|c|c|c}
\hline\hline
$N$ & 25 & 50 & 75 & 100 & 125 & 150 \\
\hline
norm & 1.412209 & 1.413691 & 1.413982 & 1.414083 & 1.414130 & 1.414156 \\
\hline\hline
$N$ & 175 & 200 & 225 & 250 & 275 & 300 \\
\hline
norm & 1.414171 & 1.414181 & 1.414188 & 1.414193 & 1.414196 & 1.414199 \\
\hline
\end{tabular}
\vspace*{.1in}
\caption{Maximum Euclidean norm for the rows of
$[\m{W}^{1/2}{\bf D}^\ddag]^{-1}$
\label{P4}}
\end{table}

\newpage

\bibliographystyle{siam}
\bibliography{library}

\begin{thebibliography}{10}

\bibitem{DarbyHagerRao11}
{\sc C.~L. Darby, W.~W. Hager, and A.~V. Rao}, {\em Direct trajectory
  optimization using a variable low-order adaptive pseudospectral method}, AIAA
  Journal of Spacecraft and Rockets, 48 (2011), pp.~433--445.

\bibitem{DarbyHagerRao10}
\leavevmode\vrule height 2pt depth -1.6pt width 23pt, {\em An hp-adaptive
  pseudospectral method for solving optimal control problems}, Optim. Control
  Appl. Meth., 32 (2011), pp.~476--502.

\bibitem{DontchevHager93}
{\sc A.~L. Dontchev and W.~W. Hager}, {\em Lipschitzian stability in nonlinear
  control and optimization}, {SIAM} J. Control Optim., 31 (1993), pp.~569--603.

\bibitem{DontchevHager97}
\leavevmode\vrule height 2pt depth -1.6pt width 23pt, {\em The {Euler}
  approximation in state constrained optimal control}, Math. Comp., 70 (2001),
  pp.~173--203.

\bibitem{DontchevHagerMalanowski00}
{\sc A.~L. Dontchev, W.~W. Hager, and K.~Malanowski}, {\em Error bounds for
  {Euler} approximation of a state and control constrained optimal control
  problem}, Numer. Funct. Anal. Optim., 21 (2000), pp.~653--682.

\bibitem{DontchevHagerVeliov00}
{\sc A.~L. Dontchev, W.~W. Hager, and V.~M. Veliov}, {\em Second-order
  {Runge}-{Kutta} approximations in constrained optimal control}, {SIAM} J.
  Numer. Anal., 38 (2000), pp.~202--226.

\bibitem{Elnagar1}
{\sc G.~Elnagar, M.~Kazemi, and M.~Razzaghi}, {\em The pseudospectral
  {Legendre} method for discretizing optimal control problems}, IEEE Trans.
  Automat. Control, 40 (1995), pp.~1793--1796.

\bibitem{Elnagar4}
{\sc G.~N. Elnagar and M.~A. Kazemi}, {\em Pseudospectral {Chebyshev} optimal
  control of constrained nonlinear dynamical systems}, Comput. Optim. Appl., 11
  (1998), pp.~195--217.

\bibitem{Fahroo2}
{\sc F.~Fahroo and I.~M. Ross}, {\em Costate estimation by a {Legendre}
  pseudospectral method}, J. Guid. Control Dyn., 24 (2001), pp.~270--277.

\bibitem{FahrooRoss02}
\leavevmode\vrule height 2pt depth -1.6pt width 23pt, {\em Direct trajectory
  optimization by a {Chebyshev} pseudospectral method}, J. Guid. Control Dyn.,
  25 (2002), pp.~160--166.

\bibitem{GargHagerRao11a}
{\sc D.~Garg, M.~A. Patterson, C.~L. Darby, C.~Fran\c{c}olin, G.~T. Huntington,
  W.~W. Hager, and A.~V. Rao}, {\em Direct trajectory optimization and costate
  estimation of finite-horizon and infinite-horizon optimal control problems
  using a {Radau} pseudospectral method}, Comput. Optim. Appl., 49 (2011),
  pp.~335--358.

\bibitem{GargHagerRao10a}
{\sc D.~Garg, M.~A. Patterson, W.~W. Hager, A.~V. Rao, D.~A. Benson, and G.~T.
  Huntington}, {\em A unified framework for the numerical solution of optimal
  control problems using pseudospectral methods}, Automatica, 46 (2010),
  pp.~1843--1851.

\bibitem{GongRossKangFahroo08}
{\sc Q.~Gong, I.~M. Ross, W.~Kang, and F.~Fahroo}, {\em Connections between the
  covector mapping theorem and convergence of pseudospectral methods for
  optimal control}, Comput. Optim. Appl., 41 (2008), pp.~307--335.

\bibitem{Hager99c}
{\sc W.~W. Hager}, {\em {Runge}-{Kutta} methods in optimal control and the
  transformed adjoint system}, Numer. Math., 87 (2000), pp.~247--282.

\bibitem{Hager02b}
\leavevmode\vrule height 2pt depth -1.6pt width 23pt, {\em Numerical analysis
  in optimal control}, in International Series of Numerical Mathematics, K.-H.
  Hoffmann, I.~Lasiecka, G.~Leugering, J.~Sprekels, and F.~Tr\"{o}ltzsch, eds.,
  vol.~139, Basel/Switzerland, 2001, Birkhauser Verlag, pp.~83--93.

\bibitem{HagerHouRao15b}
{\sc W.~W. Hager, H.~Hou, and A.~V. Rao}, {\em Convergence rate for a {Gauss}
  collocation method applied to unconstrained optimal control}, J. Optim.
  Theory Appl., submitted (2015, arxiv.org/abs/1507.08263).

\bibitem{HagerHouRao15a}
\leavevmode\vrule height 2pt depth -1.6pt width 23pt, {\em Lebesgue constants
  arising in a class of collocation methods}, IMA J.~Numer.~Anal., submitted
  (2015, arxiv.org/abs/1507.08316).

\bibitem{HJ12}
{\sc R.~A. Horn and C.~R. Johnson}, {\em Matrix Analysis}, Cambridge University
  Press, Cambridge, 2013.

\bibitem{Kameswaran1}
{\sc S.~Kameswaran and L.~T. Biegler}, {\em Convergence rates for direct
  transcription of optimal control problems using collocation at radau points},
  Comput. Optim. Appl., 41 (2008), pp.~81--126.

\bibitem{LiuHagerRao15}
{\sc F.~Liu, W.~W. Hager, and A.~V. Rao}, {\em Mesh refinement for optimal
  control using nonsmoothness detection and mesh size reduction}, J. Franklin
  Inst.,  (2015, doi:10.1016/j.jfranklin.2015.05.028).

\bibitem{NocedalWright2006}
{\sc J.~Nocedal and S.~J. Wright}, {\em Numerical Optimization}, Springer, New
  York, 2nd~ed., 2006.

\bibitem{PattersonHagerRao14}
{\sc M.~A. Patterson, W.~W. Hager, and A.~V. Rao}, {\em A $ph$ mesh refinement
  method for optimal control}, Optim. Control Appl. Meth., 36 (2015),
  pp.~398--421.

\bibitem{Reddien79}
{\sc G.~W. Reddien}, {\em Collocation at {Gauss} points as a discretization in
  optimal control}, {SIAM} J. Control Optim., 17 (1979), pp.~298--306.

\bibitem{Trefethen13}
{\sc L.~N. Trefethen}, {\em Approximation Theory and Approximation Practice},
  SIAM Publications, Philadelphia, 2013.

\bibitem{Vertesi81}
{\sc P.~V\'{e}rtesi}, {\em On {Lagrange} interpolation}, Period. Math. Hungar.,
  12 (1981), pp.~103--112.

\bibitem{Williams1}
{\sc P.~Williams}, {\em Jacobi pseudospectral method for solving optimal
  control problems}, J. Guid. Control Dyn., 27 (2004), pp.~293--297.

\end{thebibliography}
\end{document}